\documentclass[reqno,b5paper]{amsart}

\usepackage{amsmath}
\usepackage{amssymb}
\usepackage{amsthm}
\usepackage{enumerate}
\usepackage[mathscr]{eucal}
\usepackage{eqlist}

\theoremstyle{plain}
\newtheorem{theorem}{Theorem}[section]
\newtheorem{lemma}[theorem]{Lemma}

\newtheorem{step}{Step}

\theoremstyle{definition}

\newtheorem{remark}[theorem]{Remark}
\numberwithin{equation}{section}

\numberwithin{equation}{section}

\begin{document}
\title[Non-linear $\ast$- Jordan derivations on von Neumann algebras]%
{Non-linear $\ast$-Jordan derivations on von Neumann algebras}
\author[Ali Taghavi, Hamid Rohi and Vahid Darvish]%
{Ali Taghavi$^{\ast}$, Hamid Rohi and Vahid Darvish}

\newcommand{\acr}{\newline\indent}
\address{\llap{*\,}Department of Mathematics,\\ Faculty of Mathematical
Sciences,\\ University of Mazandaran,\\ P. O. Box 47416-1468,\\
Babolsar, Iran.} \email{taghavi@umz.ac.ir, h.rohi@stu.umz.ac.ir,
v.darvish@stu.umz.ac.ir}


\begin{abstract}
Let $\mathcal{A}$ be a factor von Neumann algebra and $\phi$ be the
$\ast$-Jordan derivation on $A$, that is, for every $A,B \in
\mathcal{A}$,
 $\phi(A\diamond_{1} B) = \phi(A)\diamond_{1} B + A\diamond_{1}\phi( B)$ where $A\diamond_{1} B = AB + BA^{\ast}$, then $\phi$ is additive $\ast$-derivation.
\end{abstract}

\maketitle

\section{Introduction}
Let $\mathcal{R}$ and $\mathcal{R^{'}}$ be rings. We say the map
$\Phi: \mathcal{R}\to \mathcal{R^{'}}$ preserves product or is
multiplicative if $\Phi(AB)=\Phi(A)\Phi(B)$ for all $A, B\in
\mathcal{R}$. The question of when a product preserving or
multiplicative map is additive was discussed by several authors, see
\cite{mar} and references therein. Motivated by this, many authors
pay more attention to the map on rings (and algebras) preserving
Lie product $[A,B]=AB-BA$ or Jordan product $A\circ B=AB+BA$
(for example, see \cite{bai2,bai3,hak,ji,lu,lu2,mol2,qi}). These results
show that, in some sense, Jordan product or Lie product
structure is enough to determine the ring or algebraic structure.
Historically, many mathematicians devoted themselves to the study of
additive or linear Jordan or Lie product preservers between rings or
operator algebras. Such maps are always called Jordan homomorphism
or Lie homomorphism. Here we only list several results
\cite{her,jac,kad,mar,mir,mir2}.

Let $\mathcal{R}$ be a $*$-ring. For $A,B\in\mathcal{R}$, denoted by
$A\bullet B=AB+BA^{*}$ and $[A,B]_{*}=AB-BA^{*}$, which are $\ast$-Jordan product and $\ast$-Lie product, respectively. These products are found playing a
more and more important role in some research topics, and its study
has recently attracted many author's attention (for example, see
\cite{mol,taghavi,cui,li2}).

Let define $\xi$-Jordan $\ast$-product by $A\diamondsuit_{\xi}B =
AB + \xi BA^{\ast}$. We say the map $\phi$ with property of
$\phi(A\diamondsuit_{\xi} B) = \phi(A)\diamondsuit_{\xi}B +  A
\diamondsuit_{\xi}\phi(B)$ is a $\xi$-Jordan $\ast$-derivation
map. It is clear that for $\xi = - 1$ and $\xi = 1,$ the $\xi
$-Jordan $\ast$-derivation map is a $\ast$-Lie derivation and
$\ast$-Jordan derivation, respectively \cite{bai}. We should mention here whenever  we say $\phi$ preserves derivation, it means $\phi(AB)=\phi(A)B+A\phi(B)$.\\ 
Recently, Yu and Zhang in \cite{yu}
 proved that every non-linear $\ast$-Lie derivation  from a factor von Neumann algebra into itself is an additive $\ast$-derivation. Also,
 Li, Lu and Fang
in \cite{li} have investigated a non-linear $\xi$-Jordan
$\ast$-derivation. They showed that if
$\mathcal{A}\subseteq\mathcal{B(H)}$ is a von Neumann algebra
without central abelian projections and $\xi$ is a non-zero scaler,
then $\phi:\mathcal{A} \longrightarrow \mathcal{B(H)}$ is a
non-linear $\xi$-Jordan $\ast$-derivation if and only if $\phi$ is
an additive $\ast$-derivation.

Let $\mathcal{H}$ be a complex Hilbert space and  $\mathcal{B(H)}$ be  all bounded linear operators on  $\mathcal{H}$. In this
paper we show that $\ast$-Jordan derivation map on every
factor von Neumann algebra $\mathcal{A}
 \subseteq \mathcal{B(H)}$ is
additive $\ast$-derivation.\\
 Note that a subalgebra $\mathcal{A}$ of
$\mathcal{B(H)}$ is called a von Neumann algebra when it is closed in
the weak topology of operators. A von Neumann algebra $\mathcal{A}$
is called factor when its center is trivial. It is clear that if
$\mathcal{A}$ is a factor von Neumann algebra, then $\mathcal{A}$ is
prime, that is, for $A,B \in \mathcal{A}$ if $A\mathcal{A}B =
\lbrace0\rbrace,$ then $A = 0$ or $B = 0.$ We denote real and
imaginary part of an operator $A$ by $\Re(A)$ and $\Im(A)$,
respectively i.e., $\Re(A)=\frac{A+A^{*}}{2}$ and $\Im(A)=\frac{A-A^{*}}{2i}$.
\section{The statement of the main theorem}
\noindent  The statement of our main theorem is the following.
\\
\textbf{Main Theorem.} Let $\mathcal{A}$ be a factor von Neumann algebra acting on
complex Hilbert space $\mathcal{H}$ and $\phi:\mathcal{A}
\longrightarrow \mathcal{A}$ be a $\ast$-Jordan derivation on
$\mathcal{A}$, that is, for every $A,B \in \mathcal{A}$
\begin{equation}\label{v1}
\phi(A\diamond_{1} B) = \phi(A) \diamond_{1} B + A\diamond_{1} \phi( B)
\end{equation}
where $A\diamond_{1} B = AB + BA^{\ast},$  then  $\phi$  is additive
$\ast$-derivation.
\\
\\
Before proving the Main Theorem, we need two lemmas.
\begin{lemma}\label{prop1}
Let $A \in \mathcal{A}.$ Then $ AB = -BA^{\ast}$ for every $B \in
\mathcal{A}$ implies that $ A \in \mathbb{C}I.$
\end{lemma}
\begin{proof} Let $B = I.$ We have $ A = -A^{\ast}$ and thus $AB = BA$ for every $B \in \mathcal{A}.$ Therefore,
$A \in \mathbb{C}I,$ as $\mathcal{A}$ is factor.
\end{proof}

Let $P_{1}\in \mathcal{A}$ be a non-trivial projection and $P_{2} = I
- P_{1}.$ Let $\mathcal{A}_{ij} = P_{i}\mathcal{A} P_{j}$
for $i,j = 1,2,$ we can write $\mathcal{A} =
\sum_{i,j=1,2}\mathcal{A}_{ij}$ such that their pairwise intersections are
$\{0\}$.
\\
\\
In the following Lemma we use the same idea of \cite{yu}.
\begin{lemma}\label{prop2}
 Let $A \in \mathcal{A}.$ Then $ AB = -BA^{\ast}$ for every
 $B \in \mathcal{A}_{12}$
 implies that there exists $ \lambda \in \mathbb{C}$ such that $A = \lambda P_{1} - \overline{\lambda} P_{2}.$
\end{lemma}

\begin{proof} We write  $A = A_{11} + A_{12} + A_{21} + A_{22}.$
From $AB =-BA^{\ast}$ we have $(A_{11} + A_{12} + A_{21} + A_{22})B
= -B(A^{\ast}_{11} + A^{\ast}_{12} + A^{\ast}_{21} +
A^{\ast}_{22})$. Hence, $A_{11}B + A_{21}B = -B A^{\ast}_{12} -
BA^{\ast}_{22}$, for $B \in \mathcal{A}_{12}.$ Multiplying the latter equation  by $P_{2}$ from the left side, implies that $A_{21}B = 0$
and therefore,
\begin{equation}\label{j12}
A_{21} = 0.
\end{equation}
 as $\mathcal{A}$ is
prime.\\
For every $X \in \mathcal{A}_{11}$ and $B
\in \mathcal{A}_{12}$, we can write $AXB = -XBA^{\ast}$ and also $
XAB =-XBA^{\ast}$ since $ XB$ is in $\mathcal{A}_{12}$.  Hence, $(AX -  XA)B = 0 $ so, $(AX -
XA)P_{1}TP_{2} = 0$ for every $T \in \mathcal{A}$. Thus, $(AX -
XA)P_{1} = 0$ because of primeness, so we can write $P_{1}AP_{1}X = XP_{1}AP_{1}$ since $X\in\mathcal{A}_{12}$ and $A\in\mathcal{A}$. Therefore, there
exists $\lambda \in \mathbb{C}$ such
that
\begin{equation}\label{j22}
P_{1}AP_{1} = \lambda P_{1},
\end{equation}
as $\mathcal{A}$ is factor.\\
For every $Y\in \mathcal{A}_{22}$ and $B
\in \mathcal{A}_{12},$
 we can write $ ABY = -BYA^{\ast}$ and also
 $ ABY = -BA^{\ast}Y$, since $BY$ is in $\mathcal{A}_{12}$.\\
By a similar way, we can obtain
\begin{equation}\label{j32}
 P_{2}AP_{2} = \mu P_{2},
\end{equation} 
  for some
$\mu \in \mathbb{C}.$\\
Let $A = P_{1}AP_{1} +
P_{1}AP_{2} + P_{2}AP_{1} + P_{2}AP_{2}$, Equations (\ref{j12}), (\ref{j22}) and (\ref{j32}) imply that
\begin{equation}\label{j44}
A = \lambda P_{1} + \mu P_{2}
+ P_{1}AP_{2}.
\end{equation}
 Also, From $B\in\mathcal{A}_{12}$ and Equation (\ref{j22}) we can write
$P_{1}AP_{1}B=\lambda P_{1}B=\lambda B$, it follows $P_{1}AB=\lambda B$. From the latter Equation and (\ref{j44}) we have
 $$\lambda B = P_{1}AB = -BA^{\ast} =-
\overline{\mu}B -BA^{\ast}P_{1}.$$ Multiplying above equation by $P_{2}$ from the
right side, we have $\lambda B = -\overline{\mu}B$  for every
 $B \in \mathcal{A}_{12}$.
It follows, $\mu = - \overline{\lambda}$ and so, $BA^{\ast}P_{1} = 0$ or  $BP_{2}A^{\ast}P_{1} = 0$ for
every
 $B \in \mathcal{A}_{12}$. Hence, $P_{2}A^{*}P_{1} = 0$ or $P_{1}AP_{2} = 0$. By Equation (\ref{j44}), we obtain $A = \lambda P_{1} -\overline{\lambda} P_{2},$ where $\lambda\in\mathbb{C}$. This
completes the proof of Lemma.
\end{proof}


Now we prove our Main Theorem in several Steps.
\begin{step}\label{s1}
 $\phi(0) = 0$ and $\phi(P_{i})$ are self-adjoint for $i = 1, 2.$
\end{step}
By Equation (\ref{v1}), it is easy to obtain $\phi(0) = 0.$\\
Now, we prove that $\phi(P_{i})$ are self-adjoint for $i = 1, 2.$
Let $A$ be a self-adjoint operator in $\mathcal{A}$.
Since $A\diamond_{1} P_{i}=P_{i}\diamond_{1} A$, we can write $\phi(A\diamond_{1}
P_{i})=\phi(P_{i}\diamond_{1} A)$. Then, by Equation (\ref{v1}) we have
$$\phi(A)P_{i}+P_{i}\phi(A)^{*}+A\phi(P_{i})+\phi(P_{i})A=\phi(P_{i})A+A\phi(P_{i})^{*}+P_{i}\phi(A)+\phi(A)P_{i},$$
or,
$$A(\phi(P_{i})-\phi(P_{i})^{*})=P_{i}(\phi(A)-\phi(A)^{*}).$$
We multiply above equation by $P_{j}$ from left side, it follows 
$$P_{j}A(\phi(P_{i})-\phi(P_{i})^{*})=0,$$
for all $A\in\mathcal{A}$. It means
$P_{j}\mathcal{A}(\phi(P_{i})-\phi(P_{i})^{*})= \lbrace 0 \rbrace.$
So, we have $\phi(P_{i})=\phi(P_{i})^{*}$, for $i = 1, 2$ by
primeness property of $\mathcal{A}$.
\begin{step}\label{s2}
Let $U = P_{1}\phi(P_{1})P_{2} - P_{2}\phi(P_{1})P_{1}$, we have
\end{step}
(a) for every $A \in \mathcal{A}_{12},\ \ \phi(A) = AU - UA +
P_{1}\phi (A)P_{2}$

(b) for every $B \in \mathcal{A}_{21},\  \ \phi(B) = BU - UB  +
P_{2}\phi (B)P_{1}$

(c) there exist $\alpha_{i} \in \mathbb{C}$ such that $\phi(P_{i}) =
P_{i}U - UP_{i} + \alpha _{i}P_{i}$ for every $i = 1,2.$
\\
\\
(a) Let $A \in \mathcal{A}_{12}$, we can obtain $A = P_{1}\diamond_{1} A$, by Step \ref{s1}
we have
 $$\phi(A) = \phi (P_{1})A  +  A\phi (P_{1})  +  P_{1}\phi (A)  +  \phi
 (A)P_{1}.$$
Multiplying above equation by $P_{1}$ and $P_{2}$ from two sides, respectively, we consider the following equation
have
$$P_{1}\phi(A)P_{1} = - A\phi(P_{1})P_{1},$$
$$P_{2}\phi (A)P_{2} =
P_{2}\phi (P_{1})A,$$
 \begin{equation}\label{v2}
 P_{1}\phi (P_{1})A = - A \phi(P_{1})P_{2}.
 \end{equation}
  On
the other hand, from $P_{1}\diamond_{1} P_{2} = 0$ we have
$$\phi(P_{1})P_{2} + P_{2}\phi(P_{1})  +  P_{1} \phi(P_{2}) +
\phi(P_{2})P_{1} = 0.$$ Hence, multiplying above equation by $P_{2}$ from right and left side, we have
\begin{equation}\label{vr56}
P_{2}\phi(P_{1})P_{2} = 0.
\end{equation}
Since $A\diamond_{1} P_{1} = 0,$ for $A\in\mathcal{A}_{12}$, we have $\phi(A)P_{1}  +  P_{1}\phi
(A)^{\ast}  +  A\phi(P_{1})  +  \phi(P_{1})A^{\ast} = 0.$
Multiplying the latter equation by $P_{2}$  from the left side, it
is clear that
$$P_{2}\phi(A)P_{1}  +
 P_{2}\phi(P_{1})A^{\ast} = 0$$ and it follows $$P_{2}\phi(A)P_{1}  +  P_{2}\phi(P_{1})
 P_{2}A^{\ast} P_{1} = 0,$$ since $A \in\mathcal{A}_{12}.$ The Equation (\ref{vr56}) shows $P_{2}\phi(A)P_{1} =  0.$
Hence, by assumption of $U$ and Equation (\ref{v2}) we have
\begin{eqnarray*}
\phi(A) &=& P_{2}\phi(A)P_{2}  +
P_{1}\phi(A)P_{1} + P_{2}\phi(A)P_{1}  + P_{1}\phi(A)P_{2}\\
&=&P_{2}\phi(P_{1})A  - A\phi(P_{1})P_{1}  + P_{1}\phi(A)P_{2}\\
&=& AU - UA  + P_{1}\phi(A)P_{2}.
\end{eqnarray*}
\\
(b) Let $B \in \mathcal{A}_{21}.$ From $B = P_{2}\diamond_{1} B$ and
similar to (a) we can obtain $P_{2}\phi(B)P_{2} =  -
B\phi(P_{2})P_{2}$,  $P_{1}\phi (B)P_{1} = P_{1}\phi (P_{2})B$ and $
P_{1}\phi(B)P_{2} = 0.$\\ Hence \begin{eqnarray*} \phi(B) &=&
P_{1}\phi(B)P_{1} + P_{2}\phi(B)P_{2} + P_{2}\phi(B)P_{1} +
P_{1}\phi(B)P_{2}\\
&=&P_{1}\phi(P_{2})B  - B\phi(P_{2})P_{2}  + P_{2}\phi(B)P_{1}.
\end{eqnarray*}
On the other hand, from relation $P_{2}\diamond_{1} P_{1} = 0$ we can
obtain
$$
P_{1}\phi(P_{2})P_{2} = -  P_{1}\phi(P_{1})P_{2}.$$ Multiplying
above equation by $B$ from two sides. We have $B
\phi(P_{2})P_{2} =  -  B\phi(P_{1})P_{2}$ and $P_{1}\phi(P_{2})B = -
P_{1}\phi(P_{1})B.$ Therefore, $$\phi(B) = B\phi(P_{1})P_{2}  -
P_{1}\phi(P_{1})B  + P_{2}\phi(A)P_{1}$$ and from assumption of $U$
we have
$$
\phi(B) = BU - UB  + P_{2}\phi(B)P_{1}.$$
 (c) For every  $X \in
\mathcal{A}_{11}$ and  $A \in \mathcal{A}_{12}$ and from relation
(\ref{v2}) we can write $P_{1}\phi(P_{1})XA = - XA \phi(P_{1})P_{2}$
and $ XP_{1}\phi(P_{1})A = - XA\phi(P_{1})P_{2}.$ Therefore,\\
$[P_{1}\phi(P_{1})X -  XP_{1}\phi(P_{1})]A = 0 $ and so
$[P_{1}\phi(P_{1})X - X P_{1} \phi(P_{1}) P_{1}] \mathcal{A} P_{2} =
\lbrace0\rbrace$, for $A \in \mathcal{A}_{12}$.  By the primeness of $\mathcal{A}$ and  $X \in
\mathcal{A}_{11},$ it
is clear that $ P_{1}\phi(P_{1})P_{1}X = XP_{1} \phi(P_{1})P_{1}.$\\
Since $\mathcal{A}$ is factor, $P_{1} \phi(P_{1}) P_{1} =
\alpha_{1}P_{1}$ for some $\alpha_{1} \in \mathbb{C}.$ Hence, from
Equation (\ref{vr56}) we can write
$$\phi(P_{1}) = P_{1}\phi(P_{1})P_{2} + P_{2}\phi(P_{1})P_{1} + P_{1} \phi(P_{1})P_{1} + P_{2}\phi(P_{1})P_{2} = P_{1}U - UP_{1} + \alpha
_{1}P_{1}.$$ Similar to this way, we can obtain $\phi(P_{2}) =
P_{2}U  -  UP_{2}  +  \alpha _{2}P_{2}$ for some $\alpha_{1} \in \mathbb{C}.$

\begin{remark}\label{r66}
  Let $\psi(X) = \phi(X)  -  (XU  -  UX)$ for all $X \in
\mathcal{A}.$ By a calculation, we can show that $\psi$ is
$\ast$-Jordan derivation and so, by previous Steps, $\psi(P_{i})$
are self-adjoint and 
\begin{equation}\label{j66}
\psi(P_{i}) = \alpha_{i}P_{i}
\end{equation}
 for $i = 1,2.$ Also, this shows that $\alpha_{i}$ are real.
\end{remark}
\begin{step}\label{s3}
 By assumption of $\psi,$ for every $i,j = 1,2,$ we have  $\psi(\mathcal{A}_{ij}) \subseteq \mathcal{A}_{ij}.$
\end{step}
Let $i \neq j$, Step \ref{s2} and Remark \ref{r66} show that $\psi(\mathcal{A}_{ij})
\subseteq \mathcal{A}_{ij}.$\\
Let $X \in \mathcal{A}_{ii}$, for $i = 1,2$. We have $P_{j} \diamond_{1} X = 0$, so $$\psi (P_{j})X
+ X \psi (P_{j})  +  P_{j}\psi (X)+\psi (X)P_{j} = 0,$$ since  $\psi (P_{j}) \in
\mathcal{A}_{jj},$ by Equation (\ref{j66}). Therefore,
$P_{j}\psi (X)+\psi (X)P_{j} = 0 $. Then, by multiplying the latter equation by $P_{i}$ from right and left side respectively, and $P_{j}$ from both side, we have $P_{i}\psi(X)P_{j}
= P_{j}\psi (X)P_{i} = P_{j}\psi (X)P_{j} = 0.$ Hence $\psi(X) \in \mathcal{A}_{ii}.$

\begin{step}\label{s4}
 For $i,j \in \{ 1,2\}$ with $i\neq j$, we have

(a) $\psi (A_{ii}  +  A_{jj}) = \psi (A_{ii})  +  \psi( A_{jj})$

(b) $\psi (A_{ii}  +  A_{ij}) = \psi (A_{ii})  +  \psi( A_{ij})$

(c) $\psi (A_{ii}  +  A_{ji}) = \psi (A_{ii})  +  \psi( A_{ji})$

(d) $\psi (A_{ij}  +  A_{ji}) = \psi (A_{ij})  +  \psi( A_{ji}).$\\
\end{step}
\noindent(a) From $P_{i} \diamond_{1} (A_{ii}  +  A_{jj}) = P_{i} \diamond_{1} A_{ii}$ we
have $$ \psi(P_{i}) \diamond_{1} (A_{ii}  +  A_{jj})  +  P_{i} \diamond_{1}
\psi(A_{ii}  +  A_{jj}) = \psi(P_{i}) \diamond_{1} A_{ii}  +  P_{i} \diamond_{1}
 \psi(A_{ii}),$$
So, $$ \psi(P_{i}) \diamond_{1} A_{ii}  + \psi(P_{i}) \diamond_{1} A_{jj}  + P_{i}
\diamond_{1} \psi(A_{ii}  +  A_{jj}) = \psi(P_{i}) \diamond_{1} A_{ii}  + P_{i}
\diamond_{1} \psi(A_{ii}).$$ Since $\psi(P_{i})$ is a real multiple of
$P_{i},$ by Equation (\ref{j66}). Above equation can be written as $$P_{i} \diamond_{1} \psi(A_{ii}  +  A_{jj}) =  P_{i}
\diamond_{1} \psi(A_{ii}).$$ Hence, $P_{i} \diamond_{1} K = 0$ where $K = \psi
(A_{ii}  +  A_{jj})  -  \psi (A_{ii}).$ This implies that $P_{i}K  +
KP_{i} = 0$ and so, $$P_{i}KP_{i} = P_{i}KP_{j} = P_{j}KP_{i} = 0.$$
Therefore, from $\psi(A_{ii}) \in \mathcal{A}_{ii}$ we have
\begin{equation}\label{mor}
P_{i}\psi (A_{ii}  +  A_{jj})P_{i} = \psi (A_{ii}),
\end{equation}
and 
\begin{equation}\label{mor3}
P_{i}\psi (A_{ii}  +  A_{jj})P_{j} = P_{j}\psi (A_{ii}  +
A_{jj})P_{i} = 0.
\end{equation}
A similar method shows 
\begin{equation}\label{mor2} 
 P_{j}\psi (A_{ii}  +  A_{jj})P_{j} =
 \psi (A_{jj}),
\end{equation} 
since $P_{j} \diamond_{1} (A_{ii}  + A_{jj}) = P_{j}
\diamond_{1}
 A_{jj}$. \\
 On the other hand, we can write $$\psi(A_{ii}+A_{jj})=P_{i}\psi(A_{ii}+A_{jj})P_{i}+P_{i}\psi(A_{ii}+A_{jj})P_{j}+P_{j}\psi(A_{ii}+A_{jj})P_{i}+P_{j}\psi(A_{ii}+A_{jj})P_{j}$$
So, by Equations (\ref{mor}), (\ref{mor3}) and (\ref{mor2}) we have the following
 $$\psi (A_{ii}  +  A_{jj}) = \psi (A_{ii})  +  \psi( A_{jj}).$$
(b) From $P_{j} \diamond_{1} (A_{ii}  +  A_{ij}) = P_{j} \diamond_{1} A_{ij}$ we
can write $$ \psi(P_{j}) \diamond_{1} (A_{ii}  +  A_{ij})  +  P_{j} \diamond_{1}
 \psi(A_{ii}  +  A_{ij}) = \psi(P_{j}) \diamond_{1} A_{ij}  +  P_{j} \diamond_{1}
 \psi(A_{ij}).$$
Let $ K = \psi (A_{ii}  +  A_{ij})  -  \psi (A_{ij}).$ Since
 $\psi(P_{j})$ is a real multiple of $P_{j},$ we can write $P_{j} \diamond_{1} K = 0.$ Thus, $P_{j}K  +  KP_{j} = 0$ and so
$$P_{i}KP_{j} = P_{j}KP_{i} = P_{j}KP_{j} = 0.$$
Therefore, 
\begin{equation}\label{m44}
 P_{i}\psi (A_{ii}  +  A_{ij})P_{j} = \psi (A_{ij})
\end{equation} 
and
\begin{equation}\label{m55}
P_{j}\psi (A_{ii}  +  A_{ij})P_{i} = P_{j}\psi (A_{ii}  +
A_{ij})P_{j} = 0.
\end{equation}
On the other hand, $(A_{ii}  +  A_{ij}) \diamond_{1} X_{ii} = A_{ii} \diamond_{1}
 X_{ii}$ for every $X_{ii} \in \mathcal{A}_{ii}.$ Hence, $ \psi[(A_{ii}  + A_{ij}) \diamond_{1} X_{ii}] = \psi(A_{ii} \diamond_{1} X_{ii})$ and so
$$ \psi(A_{ii}  + A_{ij}) \diamond_{1} X_{ii} +(A_{ii}  + A_{ij}) \diamond_{1}
\psi( X_{ii}) = \psi(A_{ii}) \diamond_{1} X_{ii} + A_{ii} \diamond_{1}
\psi(X_{ii}).$$ This shows $L \diamond_{1} X_{ii}= 0$ where $ L = \psi
(A_{ii}  +  A_{ij})  -
 \psi (A_{ii}).$ Thus, $ L X_{ii}=  - X_{ii}L^{\ast}$ and from Lemma \ref{prop1}. We
have $P_{i}LP_{i} = \lambda P_{i}$ for some $\lambda \in
\mathbb{C}.$ This means 
\begin{equation}\label{m66}
P_{i}\psi (A_{ii}  +  A_{ij})P_{i} =
\psi(A_{ii})  + \lambda P_{i}.
\end{equation}
Therefore, by Equations (\ref{m44}), (\ref{m55}) and (\ref{m66}), we have
 $$ \psi (A_{ii}  +
A_{ij}) = \psi(A_{ii})  +
 \psi(A_{ij})  +  \lambda P_{i}.$$
By applying this method, there exists $\alpha \in \mathbb{C}$ such that
\begin{eqnarray*}
\psi[(A_{ii}  +  A_{ij})\diamond_{1} X_{ij}] &=&
 \psi(A_{ii}X_{ij}  +  X_{ij}A^{\ast}_{ij})\\
&=& \psi(A_{ii}X_{ij})  + \psi(X_{ij}A^{\ast}_{ij})  +  \alpha
P_{i},
\end{eqnarray*}
for every $X_{ij} \in \mathcal{A}_{ij}.$ Also
\begin{eqnarray*}
\psi[(A_{ii}  +  A_{ij})\diamond_{1} X_{ij}] &=& \psi(A_{ii}  +
 A_{ij})\diamond_{1} X_{ij}  +  (A_{ii}  +  A_{ij})\diamond_{1} \psi( X_{ij})\\
&=& [\psi(A_{ii})  +  \psi(A_{ij})  +  \lambda
 P_{i}]\diamond_{1} X_{ij}  +  A_{ii}\diamond_{1} \psi( X_{ij})  +  A_{ij}\diamond_{1} \psi(
 X_{ij})\\
& =& \psi(A_{ii}\diamond_{1} X_{ij})  +  \psi(A_{ij}\diamond_{1} X_{ij})  +
\lambda P_{i}\diamond_{1} X_{ij}\\
&=& \psi(A_{ii}X_{ij})  + \psi( X_{ij}A^{\ast}_{ij})  +  \lambda
X_{ij}.
\end{eqnarray*}
Then $\alpha P_{i} = \lambda X_{ij}$ and so $\alpha P_{i} P_{i}=
 \lambda X_{ij}P_{i} = 0.$ So, $\alpha = 0$ and $\lambda = 0.$\\
 This implies that $\psi (A_{ii}  +  A_{ij}) = \psi (A_{ii})  +  \psi(
 A_{ij}).$\\
 (c) Let $X_{ji} \in \mathcal{A}_{ji}$, then,
$$\psi [(A_{ii} + A_{ji})\diamond_{1} X_{ji}] = \psi(A_{ii} + A_{ji})\diamond_{1}
X_{ji}  +  (A_{ii}  +  A_{ji})\diamond_{1} \psi (X_{ji}).$$ On the other
hand, it follows from (a)
\begin{eqnarray*} \psi [(A_{ii}  +
A_{ji})\diamond_{1} X_{ji}] &=& \psi(X_{ji}A^{\ast}_{ii} +
X_{ji}A^{\ast}_{ji})\\
&=&  \psi(X_{ji}A^{\ast}_{ii})  +
 \psi(X_{ji}A^{\ast}_{ji})\\
&=& \psi(A_{ii} \diamond_{1} X_{ji})  +  \psi(A_{ji} \diamond_{1} X_{ji})\\
&=& \psi(A_{ii}) \diamond_{1} X_{ji}  + A_{ii}\diamond_{1}\psi(X_{ji}) +
 \psi(A_{ji}) \diamond_{1} X_{ji} + A_{ji}\diamond_{1} \psi( X_{ji}) \\
&=& [\psi(A_{ii})  +  \psi(A_{ji})] \diamond_{1} X_{ji}  +  (A_{ii}  +
A_{ji}) \diamond_{1} \psi (X_{ji}).
\end{eqnarray*}
Therefore, $$\psi(A_{ii} + A_{ji})\diamond_{1} X_{ji} + (A_{ii} +
A_{ji})\diamond_{1} \psi (X_{ji}) = [\psi(A_{ii}) + \psi(A_{ji})] \diamond_{1}
X_{ji} + (A_{ii}  + A_{ji}) \diamond_{1}\psi( X_{ji}).$$ Hence, $K \diamond_{1}
X_{ji} = 0$ where $K = \psi (A_{ii}  +  A_{ji})  -
 \psi(A_{ii})  -  \psi(A_{ji}).$ So, $K X_{ji} =  - X_{ji}K^{\ast}$
for all $X_{ji} \in \mathcal{A}_{ji}.$  By using Lemma \ref{prop2},
we have
$$K =
 \alpha P_{j}  -  \overline{\alpha}P_{i}$$ for some  $\alpha \in
 \mathbb{C}.$ This implies
$$  \psi (A_{ii}  +  A_{ji}) = \psi (A_{ii})  +  \psi( A_{ji})  +  \alpha P_{j}  -
\overline{\alpha}P_{i}.$$ Since $X_{jj}\diamond_{1}   A_{ji} = X_{jj}\diamond_{1}
(A_{ii}  +  A_{ji})$ for all
 $X_{jj} \in A_{jj}$ and $\psi(X_{jj}) \in \mathcal{A}_{jj},$ we can write
 \begin{eqnarray*}
 \psi(X_{jj})\diamond_{1} A_{ji} + X_{jj}\diamond_{1} \psi( A_{ji}) &=&
 \psi(X_{jj})\diamond_{1} (A_{ii}  +  A_{ji})  +  X_{jj}\diamond_{1} \psi(A_{ii}  +
 A_{ji})\\
&=& \psi(X_{jj})\diamond_{1} A_{ji} +  X_{jj}\diamond_{1} \psi(A_{ii}  +
A_{ji})\\
&=& \psi(X_{jj})\diamond_{1} A_{ji} +  X_{jj}\diamond_{1} [ \psi (A_{ii})  + \psi(
A_{ji})  +  \alpha P_{j}  -  \overline{\alpha}P_{i}]\\
&=& \psi(X_{jj})\diamond_{1} A_{ji}  +  X_{jj}\diamond_{1} \psi(A_{ji})  +
X_{jj}\diamond_{1} \alpha
 P_{j}.
\end{eqnarray*}
So, $X_{jj}\diamond_{1} \alpha P_{j} = 0.$ Thus, $\alpha X_{jj} =  - \alpha
X^{\ast}_{jj}$ for all $X_{jj} \in
 \mathcal{A}_{jj}$ and so $\alpha = 0.$ Hence,
$$\psi (A_{ii}  +  A_{ji}) = \psi (A_{ii})  +  \psi( A_{ji})$$
(d) From the fact $P_{i} \diamond_{1} (A_{ij}  +  A_{ji}) =A_{ij}  +
A_{ji},$ we have
 $$\psi(P_{i})(A_{ij} +  A_{ji})  +  (A_{ij}  +  A_{ji})\psi(P_{i})  +  P_{i}\psi(A_{ij}  +  A_{ji})  +  \psi(A_{ij}  +  A_{ji})P_{i} = \psi(A_{ij}  +
 A_{ji}).$$
Multiplying above equation by $P_{j}$ from two sides, we have
$$P_{j}\psi(A_{ij} + A_{ji})P_{j} = 0.$$
Similarly, from $P_{j} \diamond_{1} (A_{ij}  +  A_{ji}) =A_{ij}  +
A_{ji},$ we have $P_{i}\psi(A_{ij}  +  A_{ji})P_{i} = 0.$\\
On the other hand, from $(A_{ij} + A_{ji}) \diamond_{1} P_{i} = A_{ji}\diamond_{1}
P_{i}$ we have $$\psi(A_{ij}  +  A_{ji}) \diamond_{1} P_{i}  + (A_{ij}  +
A_{ji})\diamond_{1} \psi( P_{i}) =  \psi(A_{ji})\diamond_{1} P_{i}  + A_{ji}\diamond_{1}
\psi(P_{i}).$$ This implies that $K \diamond_{1} P_{i} = 0 $ where $K =
\psi (A_{ij}  +
 A_{ji})  -  \psi( A_{ji}).$ So, $ KP_{i}  +  P_{i}K^{\ast} = 0.$ Thus $ P_{j}KP_{i} = 0$ and so $ P_{j}\psi(A_{ij}  +  A_{ji})P_{i} =
 \psi(A_{ji}).$\\
Similarly, from $ (A_{ij} + A_{ji})\diamond_{1} P_{j} = A_{ij}\diamond_{1} P_{j},$
we can obtain $P_{i}\psi(A_{ij}  +  A_{ji})P_{j} = \psi(A_{ij}).$\\
These relations show that $$\psi (A_{ij}  +  A_{ji}) = \psi (A_{ij})
+  \psi( A_{ji}).$$
\begin{step}\label{s5} For $1\leq i \neq j \leq2,$ we have $\psi (\sum_{i,j = 1,2}A_{ij}) =  \psi(A_{ii}) + \psi(A_{ij}) + \psi(A_{ji}) + \psi(A_{jj}).$
\end{step}
First we show that $$ \psi (A_{ii} + A_{ij} + A_{ji}) = \psi(A_{ii})
+ \psi(A_{ij}) + \psi(A_{ji}).$$ From $P_{j}\diamond_{1} (A_{ii} + A_{ij} +
A_{ji}) = P_{j}\diamond_{1} ( A_{ij} + A_{ji})$ and part (d) of Step
\ref{s4} we have
$$\psi(P_{j})\diamond_{1} (A_{ii} + A_{ij} + A_{ji}) + P_{j}\diamond_{1} \psi(A_{ii} + A_{ij} + A_{ji})  =
\psi(P_{j})\diamond_{1} ( A_{ij} + A_{ji}) + P_{j}\diamond_{1}(\psi ( A_{ij})
+\psi( A_{ji})). $$ So, $ P_{j}\diamond_{1} \psi(A_{ii} + A_{ij} + A_{ji})
=
P_{j}\diamond_{1}[\psi ( A_{ij}) +\psi( A_{ji})].$\\
 Hence, $P_{j}\diamond_{1} K =
0 $ where $ K = \psi (A_{ii} + A_{ij} + A_{ji}) - \psi(A_{ij}) -
\psi(A_{ji}).$ Then $P_{j} K + KP_{j} = 0 $ and so, $P_{j} K P_{j} =
P_{i}KP_{j} =P_{j} K P_{i}= 0. $ Thus we have
$$P_{j} \psi (A_{ii} +
A_{ij} + A_{ji}) P_{j}= 0,$$ $$P_{i}\psi (A_{ii} + A_{ij} + A_{ji})
P_{j} = \psi(A_{ij})$$ and $$P_{j}\psi (A_{ii} + A_{ij} + A_{ji})
P_{i} = \psi(A_{ji}).$$ On the other hand, from $(A_{ii} + A_{ij} +
A_{ji}) \diamond_{1} T_{ii} = (A_{ii} + A_{ji}) \diamond_{1} T_{ii} $ and part (c)
of Step \ref{s4} we have $L\diamond_{1} T_{ii} = 0 $ where $L  = \psi
(A_{ii} + A_{ij} + A_{ji}) - \psi(A_{ii}) - \psi(A_{ji}).$ Thus, $ L
T_{ii} + T_{ii}L^{\ast}= 0.$ By Lemma \ref{prop1}, the latter
equation yields $P_{i}LP_{i} =\alpha P_{i}$ for some $\alpha \in
\mathbb{C}.$ Thus,
$$P_{i}\psi (A_{ii} + A_{ij} + A_{ji}) P_{i} = \psi (A_{ii}) + \alpha
P_{i}.$$ Hence, we obtain
 $$\psi (A_{ii} + A_{ij} + A_{ji}) = \psi (A_{ii}) + \psi (A_{ij}) + \psi (A_{ji}) + \alpha
 P_{i}.$$
We will show that $\alpha = 0.$ From above relation, for every
$T_{ii} \in \mathcal{A}_{ii},$ there exists $\lambda \in \mathbb{C}$
such that
\begin{eqnarray}
\psi[T_{ii} \diamond_{1} (A_{ii} + A_{ij} + A_{ji})]& =&  \psi(T_{ii}A_{ii}
+ T_{ii}A_{ij} +A_{ii}T^{\ast}_{ii} + A_{ji}T^{\ast}_{ii}) \nonumber\\
&=&  \psi(T_{ii}A_{ii} +A_{ii}T^{\ast}_{ii}) +
\psi(T_{ii}A_{ij})\nonumber\\
&& + \psi (A_{ji}T^{\ast}_{ii}) + \lambda P_{i}.\label{eq4}
\end{eqnarray}
On the other hand,
\begin{eqnarray*}
\psi[T_{ii} \diamond_{1} (A_{ii} + A_{ij} + A_{ji})]&=&  \psi(T_{ii}) \diamond_{1}
(A_{ii} + A_{ij} + A_{ji}) + T_{ii} \diamond_{1} \psi (A_{ii} + A_{ij} +
A_{ji})\\
&=&  \psi(T_{ii})  \diamond_{1} (A_{ii} + A_{ij} + A_{ji}) + T_{ii} \diamond_{1} (\psi (A_{ii}) +\psi (A_{ij})\\
&&+\psi (A_{ji}) + \alpha P_{i})\\
&=&  \psi(T_{ii}) \diamond_{1} A_{ii} + T_{ii}\diamond_{1}\psi (A_{ii}) +
\psi(T_{ii}) \diamond_{1} A_{ij} + T_{ii}\diamond_{1}\psi (A_{ij}) \\
&&+ \psi(T_{ii}) \diamond_{1} A_{ji} + T_{ii}\diamond_{1}\psi (A_{ji}) + T_{ii}
\diamond_{1} \alpha
P_{i}\\
&=&  \psi(T_{ii} \diamond_{1} A_{ii}) +  \psi(T_{ii} \diamond_{1} A_{ij})  +
\psi(T_{ii} \diamond_{1} A_{ji}) + \alpha(T_{ii} + T^{\ast}_{ii})\\
&=& \psi(T_{ii}A_{ii} + A_{ii}{T_{ii}}^{*})  + \psi(T_{ii} A_{ij}) +
\psi(A_{ji}{T_{ii}}^{*})+ \alpha (T_{ii}\\
&& + {T_{ii}}^{*}).
\end{eqnarray*}
Therefore, from relation (\ref{eq4}) we have $\lambda  P_{i} =
\alpha(T_{ii} +
T^{\ast}_{ii})$  for every $ T_{ii} \in \mathcal{A}_{ii}.$\\
Thus $\lambda = \alpha = 0$ and finally we have
\begin{equation}\label{equ5}
 \psi (A_{ii} +
A_{ij} + A_{ji}) =  \psi(A_{ii}) + \psi(A_{ij}) + \psi(A_{ji}).
\end{equation}
Now, we prove 
$$  \psi (A_{ii}+A_{ij}+A_{ji}+A_{jj}) =  \psi(A_{ii}) + \psi(A_{ij}) + \psi(A_{ji}) +
\psi(A_{jj}).$$
Since $P_{i} \diamond_{1}  (\sum_{i,j = 1,2}A_{ij}) = P_{i} \diamond_{1} (A_{ii} +
A_{ij} + A_{ji})$, we have $$\psi(P_{i}) \diamond_{1}  (\sum_{i,j =
1,2}A_{ij}) +P_{i} \diamond_{1}  \psi(\sum_{i,j = 1,2}A_{ij})
 = \psi(P_{i}) \diamond_{1} (A_{ii} + A_{ij} + A_{ji}) + P_{i} \diamond_{1} \psi(A_{ii} + A_{ij} +
 A_{ji}).$$
From relation (\ref{equ5}) we have $P_{i} \diamond_{1} M =0$ where $$M =
\psi (\sum_{i,j = 1,2}A_{ij}) -  \psi(A_{ii}) - \psi(A_{ij}) -
\psi(A_{ji}).$$ Hence, $P_{i} M + MP_{i} =0.$ This implies that
$P_{i}MP_{i} = P_{i}MP_{j}  = P_{j}MP_{i} = 0.$ Therefore,

$$P_{i}\psi (\sum_{i,j = 1,2}A_{ij})P_{i} =  \psi(A_{ii}),$$$$ P_{i}\psi (\sum_{i,j = 1,2}A_{ij})P_{j} =
\psi(A_{ij})$$ and $$ P_{j}\psi (\sum_{i,j = 1,2}A_{ij})P_{i} =
\psi(A_{ji}).$$
 Since $P_{j} \diamond_{1} (\sum_{i,j = 1,2}A_{ij}) = P_{j}
\diamond_{1}   (A_{ij} + A_{ji} + A_{jj})$  by a similar method, we have $$
P_{j}\psi (\sum_{i,j = 1,2}A_{ij})P_{j} =  \psi(A_{jj}).$$
 Finally,
we have
$$  \psi (\sum_{i,j = 1,2}A_{ij}) =  \psi(A_{ii}) + \psi(A_{ij}) + \psi(A_{ji}) +
\psi(A_{jj}).$$
\begin{step}\label{s6}
$\psi (A_{ij} + B_{ij}) = \psi (A_{ij}) + \psi( B_{ij})$ for every
$A_{ij} , B_{ij} \in \mathcal{A}_{ij}$ such that $ i,j = 1,2.$
\end{step}
Let $i \neq j$, then $ T_{ij} + T^{\ast}_{ij} = T_{ij}\diamond_{1} P_{j},$
and so, by part (d) of Step \ref{s4}, we have

 \begin{eqnarray*} \psi(T_{ij}) +
\psi(T^{\ast}_{ij}) &=& \psi(T_{ij})\diamond_{1} P_{j} + T_{ij}\diamond_{1}
\psi(P_{j})\\
&=& \psi(T_{ij}) + \psi(T_{ij})^{\ast} + T_{ij} \psi(P_{j}) +
\psi(P_{j})T^{\ast}_{ij}. \end{eqnarray*}
 Multiplying above equation by $P_{j}$
from the right side, we have $ T_{ij} \psi(P_{j}) = 0$ for every
$T_{ij} \in \mathcal{A}_{ij}.$ So, $ \psi(P_{j}) = 0$ for $j =
1,2.$\\
Let $A_{ij} , B_{ij} \in \mathcal{A}_{ij}$ such that $(i\neq j)$.
Then,
$$ A_{ij}
+ B_{ij} + A^{\ast}_{ij} + B_{ij}A^{\ast}_{ij} = (P_{i} + A_{ij})
\diamond_{1}(P_{j} + B_{ij}),$$ and so, by Steps \ref{s4} and \ref{s5}, we
have
\begin{eqnarray*}
\psi (A_{ij} + B_{ij}) + \psi ( A^{\ast}_{ij}) +
\psi(B_{ij}A^{\ast}_{ij})&=& \psi( A_{ij} + B_{ij}+ A^{\ast}_{ij} +
B_{ij}A^{\ast}_{ij})\\
&=& \psi[(P_{i} + A_{ij}) \diamond_{1}(P_{j} + B_{ij})]\\
&=&\psi(P_{i} + A_{ij}) \diamond_{1}(P_{j} + B_{ij})\\
&& + (P_{i}+
A_{ij})\diamond_{1} \psi(P_{j} + B_{ij})\\
&=& [\psi(P_{i}) +\psi (A_{ij})] \diamond_{1}(P_{j} + B_{ij})\\
&& + (P_{i} +
A_{ij}) \diamond_{1} [\psi(P_{j}) + \psi (B_{ij})]\\
& =& \psi (A_{ij}) \diamond_{1}(P_{j} + B_{ij}) + (P_{i} + A_{ij}) \diamond_{1}
\psi (B_{ij})\\
&=& \psi (A_{ij}) + \psi (B_{ij}) + \psi ( A_{ij})^{\ast} +
B_{ij}\psi (A_{ij})^{\ast}\\
&& + \psi (B_{ij})A^{\ast}_{ij}.
\end{eqnarray*}
Multiplying by $P_{j}$ from the right side implies that $\psi
(A_{ij} + B_{ij}) = \psi (A_{ij}) + \psi( B_{ij})$ for every
$A_{ij} , B_{ij} \in \mathcal{A}_{ij}$  such that $i\neq j.$\\
Let $A_{ii} , B_{ii} \in \mathcal{A}_{ii}$ and $T_{ij}
\in\mathcal{A}_{ij}.$ It follows from above relation that
\begin{eqnarray*}
\psi[(A_{ii} + B_{ii})\diamond_{1} T_{ij}] &=&  \psi(A_{ii}T_{ij}+
B_{ii}T_{ij})\\
&=& \psi(A_{ii}T_{ij}) + \psi (B_{ii}T_{ij})\\
&=&\psi(A_{ii}\diamond_{1} T_{ij}) + \psi(B_{ii}\diamond_{1} T_{ij})\\
&=&\psi(A_{ii})\diamond_{1} T_{ij} + A_{ii}\diamond_{1} \psi (T_{ij}) +
\psi(B_{ii})\diamond_{1} T_{ij} + B_{ii}\diamond_{1} \psi(T_{ij})\\
&=&\psi(A_{ii})T_{ij} + A_{ii}\psi(T_{ij}) + \psi(B_{ii})T_{ij} +
B_{ii}\psi(T_{ij}).
\end{eqnarray*}
So, 
$$\psi[(A_{ii} + B_{ii})\diamond_{1} T_{ij}]=\psi(A_{ii})T_{ij} + A_{ii}\psi(T_{ij}) + \psi(B_{ii})T_{ij} +
B_{ii}\psi(T_{ij}).$$
On the other hand, since $\psi(A_{ii} + B_{ii})\in\mathcal{A}_{ii}$ and above equation, we have
\begin{eqnarray*}
 \psi[(A_{ii} + B_{ii})\diamond_{1} T_{ij}] &=& \psi(A_{ii} + B_{ii})\diamond_{1}
T_{ij} + (A_{ii} + B_{ii})\diamond_{1} \psi (T_{ij})\\
&=& \psi(A_{ii} +B_{ii})T_{ij} + A_{ii}\psi(T_{ij}) +
B_{ii}\psi(T_{ij}).
\end{eqnarray*}
Hence, $[\psi(A_{ii} + B_{ii}) -  \psi(A_{ii}) -  \psi(B_{ii})]
T_{ij} = 0$ for every $T_{ij} \in \mathcal{A}_{ij}.$ This implies
that $\psi(A_{ii} + B_{ii}) =  \psi(A_{ii}) +  \psi(B_{ii}).$
\begin{step}\label{s7}
$\psi$ is additive and $\ast$-preserving on $\mathcal{A}.$
\end{step}
Let $A = \sum_{i,j = 1,2}A_{ij} $ and $B = \sum_{i,j = 1,2}B_{ij} $
for every $A , B \in \mathcal{A},$ then from Steps \ref{s5} and
\ref{s6} we have
\begin{eqnarray*}
\psi(A+B) &=& \psi( \sum_{i,j = 1,2}A_{ij} + \sum_{i,j =
1,2}B_{ij})\\ &=& \psi( \sum_{i,j = 1,2}(A_{ij} + B_{ij}))\\
 &=&
\sum_{i,j = 1,2}\psi(A_{ij} + B_{ij})\\
 &=& \sum_{i,j =
1,2}\psi(A_{ij}) +\sum_{i,j = 1,2}\psi (B_{ij})\\
 &=& \psi( \sum_{i,j
= 1,2} A_{ij}) + \psi (\sum_{i,j = 1,2} B_{ij}) = \psi(A)+\psi (B).
\end{eqnarray*}
Then $\psi$ is additive.\\

Now, we will prove $\psi$ is $\ast$-preserving. We showed in
Step \ref{s6} that $\Phi(P_{i})=0$ for $i=1,2$. So, $\psi(I) =
\Phi(P_{1}) + \Phi(P_{2}) = 0$. So,
$$\psi(A\diamond_{1} I)=\psi(A)\diamond_{1} I+A\diamond_{1} \psi(I),$$
for all $A\in \mathcal{A}$, yields
$\psi(A+A^{*})=\psi(A)+\psi(A)^{*}$. Since $\psi$ is additive we
have $\psi(A)+\psi(A^{*})=\psi(A)+\psi(A)^{*}$, for all $A\in
\mathcal{A}$, and so $\psi(A^{*})= \psi(A)^{*}$. Thus $\psi$ is
$\ast$-preserving.

\begin{step}\label{s8}
 $\psi(AB) = \psi(A)B + A\psi(B) $ for every $A,B\in \mathcal{A}.  $
\end{step}
Here, we prove our Step by three cases.\\
Case 1. Let $A^{\ast} = -A $ (skew self-adjoint) and $B^{\ast} = B$.\\
By our Main Theorem assumption we have
$$\psi(A\diamond_{1} B) = \psi(A)\diamond_{1} B + A\diamond_{1}\psi(B)$$ and
$$\psi(B\diamond_{1} A) = \psi(B)\diamond_{1} A + B\diamond_{1}\psi(A).$$
By Step \ref{s7}, we know $\psi$ is $ \ast- $preserving, i.e.,
$\psi(T^{\ast}) = \psi(T)^{\ast}$ for every $ T\in \mathcal{A}$. So,
from above relation we have the following
$$\psi(AB - BA) = \psi(A)B - B\psi(A) + A\psi(B) - \psi(B)A,$$
and
$$\psi(BA + AB) = \psi(B)A + A\psi(B) + B\psi(A) + \psi(A)B.$$
Adding these relations, by additivity of $ \psi $, we have
 \begin{equation}\label{v56}
\psi(AB) = \psi(A)B +A\psi(B)
\end{equation}
for $A^{\ast} = -A $ and $B^{\ast} = B.$ It means $\psi$ is
derivation for skew self-adjoint $A$ and self-adjoint $B$.\\
Case 2. Let $A$ and $B$ be self-adjoint.\\
Before we prove $\psi$ is derivation for self-adjoint operators, we
need to show $\psi(iA)=i\psi(A)$, for all $A\in\mathcal{A}$. For
this purpose we should verify $\psi(iI)=0$.\\
We have $iT_{21} + iT^{\ast}_{21} = T_{21}\diamond_{1} iP_{1}$ for all $
T_{21}\in \mathcal{A}_{21}$, and so, by additivity of $\psi$ and
getting $\psi$ of the latter equation, we have
$$\psi( iT_{21}) + \psi(iT^{\ast}_{21}) = i\psi(T_{21})P_{1} + iP_{1}\psi(T_{21})^{\ast} + T_{21}\psi(iP_{1}) +  \psi(iP_{1})T^{\ast}_{21}.$$
Multiplying above equation by $P_{1}$ from the right side and also
from $ \psi(T_{21})\in \mathcal{A}_{21},$ we have
\begin{equation}\label{v432}
 \psi( iT_{21}) = i\psi(T_{21}) + T_{21}\psi(iP_{1}).
\end{equation}
Let put $iT_{21}$ instead of $T_{21}$ in Equation (\ref{v432}) we
have
\begin{equation}\label{v67}
-\psi(T_{21}) = i\psi(iT_{21}) + iT_{21}\psi(iP_{1}),
\end{equation}
and multiply Equation (\ref{v432}) by $i$, we have
\begin{equation}\label{v68}
i\psi(iT_{21}) = -\psi(T_{21}) + iT_{21}\psi(iP_{1}).
\end{equation}
 Adding Equations (\ref{v67}) and (\ref{v68}) together, we have $ 2iT_{21}\psi(iP_{1}) = 0 $, for all
$T_{21}\in\mathcal{A}_{21}$. We obtain $\psi(iP_{1}) = 0$,
since $A_{21}$ is prime.\\
 By a similar way, we can obtain $\psi(iP_{2}) = 0$. Then, from additivity of $ \psi, $
$$\psi(iI) = \psi(iP_{1}) +\psi(iP_{2}) = 0.$$
Now, we are ready to show $\psi(iA)=i\psi(A)$, for all
$A\in\mathcal{A}$.\\
By applying Equation (\ref{v56}) and the fact that $\psi(iI)=0$, we
have
 \begin{equation}\label{v76}
\psi(iA) = \psi(iIA) =\psi(iI)A + iI\psi(A) = i\psi(A),
\end{equation} for all self-adjoint operators $A\in\mathcal{A}$. It
is easy to see that we have $\psi(iA)=i\psi(A)$, for all
$A\in\mathcal{A}$, since we can write as follow by additivity of
$\psi$ and Equation (\ref{v76})
\begin{eqnarray*}
\psi(iA)&=&\psi(i(\Re(A)+i\Im(A)))\\
&=&\psi(i\Re(A)-\Im(A))\\
&=&\psi(i\Re(A))-\psi(\Im(A))\\
&=&i[\psi(\Re(A))+i\psi(\Im(A))]\\
&=& i\psi(\Re(A)+i\Im(A))=i\psi(A).
\end{eqnarray*}
So, we proved that
\begin{equation}\label{vv23}
\psi(iA)=i\psi(A),
\end{equation}
for all $A\in\mathcal{A}$. \\
Now, let get back to prove $\psi$ is derivation for self-adjoint
operators $ A, B \in\mathcal{A}$. By Equation (\ref{v1}) we have the
following \begin{equation}\label{v478}
\psi(AB+BA)=\psi(A)B+B\psi(A)+A\psi(B)+\psi(B)A,
\end{equation}
for all self-adjoint operators $A, B\in\mathcal{A}$.\\
On the other hand, by applying Equation (\ref{v1}) for $iB$ and $iA$
we have
 $\psi(iB\diamond_{1} iA) = \psi(iB)\diamond_{1} iA + iB\diamond_{1} \psi(iA).$\\
Hence, by Equation (\ref{v76}) we have
\begin{eqnarray*}
\psi(-BA + AB) &=& \psi(iB)iA + iA\psi(iB)^{\ast} + iB\psi(iA) +
\psi(iA)(iB)^{\ast}\\
&=& -\psi(B)A + A\psi(B) - B\psi(A) + \psi(A)B.
\end{eqnarray*}
So,
\begin{equation}\label{v345}
\psi(-BA + AB)=-\psi(B)A + A\psi(B) - B\psi(A) + \psi(A)B,
\end{equation}
for all self-adjoint $A, B\in\mathcal{A}$. Adding Equations
(\ref{v478}) and (\ref{v345}) we have \begin{equation}\label{fv1}
\psi(AB) = \psi(A)B + A\psi(B)
\end{equation}
 for self adjoint operators in $\mathcal{A}$.\\
Case 3. Finally, we prove $\psi$ is derivation for all $A$ and $B$
in $\mathcal{A}$. By Equation (\ref{fv1}) we have the following
\begin{eqnarray*}
\psi(AB) &= & \psi[(\Re(A) + i\Im(A))(\Re(B) + i\Im(B))]\\
& =& \psi(\Re(A)\Re(B)) + i\psi(\Re(A)\Im(B)) + i\psi(\Im(A)\Re(B))
- \psi(\Im(A)\Im(B))\\
& =& \psi(\Re(A)) \Re(B) + \Re(A)\psi(\Re(B)) + i\psi(\Re(A))\Im(B)
+ i\Re(A)\psi(\Im(B)) \\
&&+ i\psi(\Im(A))\Re(B) + i\Im(A)\psi(\Re(B)) -
\psi(\Im(A))\Im(B) - \Im(A)\psi(\Im(B))\\
& =& \psi(\Re(A))\Re( B) + \Re(A)\psi(\Re(B)) + i\psi(\Re(A))\Im(B)
+ i\Re(A)\psi(\Im(B)) \\
&&+ i\psi(\Im(A))\Re(B) + i\Im(A)\psi(\Re(B)) +
\psi(i\Im(A))i\Im(B) + i\Im(A)\psi(i\Im(B))\\
& =& \psi(\Re(A))[\Re(B) + i\Im(B)] + \Re(A)\psi[\Re(B) + i\Im(B)]\\
&& +
i\psi(\Im(A))[\Re(B) + i\Im(B)] + i\Im(A)\psi[\Re(B) + i\Im(B)]\\
& =& \psi(\Re(A))B + \Re(A)\psi(B) + i\psi(\Im(A))B +
i\Im(A)\psi(B)\\
& =& \psi(A)B + A\psi(B).
\end{eqnarray*}

This completes the proof of main Theorem.

\bibliographystyle{amsplain}

\end{document}